\documentclass[12pt]{amsart}
\usepackage{amsfonts,amsmath,amssymb,amscd,amssymb}
\usepackage{enumerate}
\usepackage{enumitem}
\usepackage{comment}
\usepackage{hyperref}
\newtheorem{theorem}{Theorem}[section]
\newtheorem{lemma}[theorem]{Lemma}
\newtheorem{corollary}[theorem]{Corollary}

\newtheorem{remark}[theorem]{Remark}
\newtheorem{proposition}[theorem]{Proposition}

\newcommand{\Z}{\mbox{$\mathbb Z$}}
\newcommand{\Q}{\mbox{$\mathbb Q$}}
\newcommand{\N}{\mbox{$\mathbb N$}}     
\newcommand{\C}{\mbox{$\mathbb C$}}     

\begin{document}
	\title[Diophantine Equations for Polynomial Recursive Sequences]{Diophantine Equations for Polynomial Recursive Sequences}
	
	\author[Darsana]{Darsana N}
	\address{Darsana N, Department of Mathematics, National Institute of Technology Calicut, 
		Kozhikode-673 601, India.}
	\email{darsana\_p230059ma@nitc.ac.in; darsanasrinarayan@gmail.com}

	\author[Rout]{S. S. Rout}
	\address{Sudhansu Sekhar Rout, Department of Mathematics, National Institute of Technology Calicut, 
		Kozhikode-673 601, India.}
	\email{sudhansu@nitc.ac.in; lbs.sudhansu@gmail.com}

	\dedicatory{}
	\thanks{2020 Mathematics Subject Classification: 11B37 (Primary), 11D61, 11R58 (Secondary).\\
		Keywords: Polynomial recurrence sequences, Diophantine equations, $S$-units, function fields, Brownawell-Masser inequality}
	\begin{abstract}
        We study the Diophantine equation of type $U_n(x)=V_m(y)$, where $(U_n)_{n\geq 0}$ and $(V_m)_{m\geq 0}$ are polynomial power sums defined over a number field $K$. By applying the finiteness criterion of Bilu and Tichy, we show under appropriate assumptions that equation $U_n(x)=V_m(y)$ has infinitely many solutions with bounded $\mathcal{O}_S$-denominator. We also study decomposable polynomials in third and second order linear recurrence sequences. In particular, we show that if $W_n(x)=g(h(x))$ for a simple third order linear recurrence sequence $(W_n(x))_{n\geq 0}$ of complex polynomials, then $\deg g$ is bounded. Furthermore, we show that if $(u_{n_1}+u_{n_2})(x)=g(h(x))$ for a binary recurrence sequence $(u_n(x))_{n\geq 0}$ then $\deg g$ is bounded.
	\end{abstract}
	\maketitle
	\pagenumbering{arabic}
	\pagestyle{headings}

	\section{Introduction}
	The study of polynomial Diophantine equations of    the form
	\begin{equation}\label{eq1}
		f(x)=g(y)
	\end{equation} 
    have been of interest to many number theorists. Observe that a defining equation of an elliptic curve $y^2=x^3+ax+b$ with $a,b\in \Q$, $4a^3+27b^2\neq 0$ is an example of such equation.  The essential question concerning \eqref{eq1} is whether there exist finitely or infinitely many integral solutions of \eqref{eq1} in $S$-integers $x$ and $y$. By Siegel’s classical theorem, it follows that an irreducible algebraic curve defined over a number field can have only finitely many $S$-integral points, unless it has genus zero and at most two points at infinity. However, this approach is ineffective since it does not provide explicit bounds for the size of the solutions $(x,y)$.  Over the years, many authors such as Baker, Sprindžuk, Brindza, Bugeaud, and others developed effective versions of this finiteness result, especially for particular classes of equations (see for example \cite{baker1969bounds}, \cite{sprindzuk1982classical}). A major advancement in this direction was made by Bilu and Tichy \cite{bilu2000diophantine}, who established precise conditions under which \eqref{eq1} has infinitely many $S$-integer solutions with a bounded denominator. Their work marked a major step forward  in the study of Diophantine equations of the type \eqref{eq1}. It turns out to be more convenient to investigate a slightly more general version of the problem.

   One well-known problem involving the Diophantine equations with linear recurrences is to determine how many zeros occur in such a sequence, or more generally, to estimate the number of indices $n \in \mathbb{N}$ satisfying the equation $G_n(x) = a$ for a given $a$ in a function field (see \cite{fuchs2005effective}). Glass, Loxton, and van der Poorten \cite{glass1981identifying} proved that if the sequence $(G_n(x))_{n=0}^{\infty}$ is nonperiodic and nondegenerate, 
  then there exist only finitely many pairs of   integers $m, n$ with $m>n\geq 0$ such that
  \begin{equation*}
    G_n(x) = G_m(x).
  \end{equation*} This result has generalized to an equation of the form
  \begin{equation}\label{eqa1}
        G_n(x) = G_m(P(x)),
   \end{equation} where $P(x)\in K[x]$ (see \cite{fuchs2002diophantine}).
   In \cite{dujella2001diophantine}, Dujella and Tichy investigated the family of linear recurrence sequences defined by
   $G_{n+1}(x) = x G_n(x) + B G_{n-1}(x), \quad G_0(x) = 0, \quad G_1(x) = 1,$ 
    where $B$ is a nonzero integer. They proved that for such sequences, there exist no polynomial $P(x) \in \mathbb{C}[x]$ satisfying \eqref{eqa1}
    whenever $m,n \ge 3$ and $m \neq n$. As a continuation of this, Fuchs \cite{fuchs2004diophantine} obtained conditions under which the Diophantine equation \eqref{eqa1} has finitely many solutions in $n$ and $m$ for a third order linear recurrence. Moreover, Fuchs et al. \cite{fuchs2003diophantine} have generalized this result for an arbitrary degree recurrence. 
     By applying the result due to Bilu and Tichy \cite{bilu2000diophantine}, these findings led to the conclusion that the Diophantine equation
    \[
     G_n(x) = G_m(y)
    \]
    admits only finitely many integer solutions in $n,m,x,y$ with $n \neq m$. Further, considering simple linear recurrences of polynomials,  Fuchs and Heintze \cite{fuchs2021diophantine} examined the Diophantine equation
\begin{equation*}\label{eq3}
		G_n(x)=H_m(y)
\end{equation*} and proved that it admits infinitely many rational solutions with a bounded denominator only when certain specific identity hold.
In \cite{kreso2017diophantine}, Kreso has given a similar result concerning \eqref{eq1} when $f$ and $g$ are lacunary polynomial. By analyzing the properties of monodromy group of a polynomial, Kreso and Tichy have given a generalized result in \cite{kreso2018diophantine}.  
	
	
The theory of functional decomposition $f = f_1 \circ \cdots \circ f_m$ of a complex polynomial $f$ into indecomposable complex polynomials $f_1,\ldots, f_m$ was developed by J. F. Ritt \cite{ritt}, where a complex polynomial $f$ is said to be
indecomposable if $\deg f > 1$ and it cannot be represented as a composition of two lower
degree polynomials. Several authors have investigated the structure and properties of decomposable polynomial. The decomposability among terms of second-order linear recurrence sequences has been studied under some assumptions (see \cite{fuchs2019decomposable,kreso2022decompositions}). The complete decomposition of linear combinations
of the Bernoulli and Euler polynomial has been studied (see \cite{pinter, pinter2022, pinter2024}).
	
In this paper, firstly, we extend a result of Fuchs and Heintze \cite{fuchs2021diophantine} to the case of polynomial sequences defined by a polynomial recurrence relation over a number field. Secondly, we extend the main result of \cite{fuchs2019decomposable}  for a third order polynomial recurrence. Further, we study the decomposition properties of
sum of two terms of a binary recurrence sequences of polynomials. In Section \ref{sec:notation}, we give the required notation and state our main results. In Section \ref{sec:auxiliary}, we describe finiteness criterion proposed by Bilu and Tichy, which is essential to prove one of our result. We define projective height over a function field and state some properties of height functions. Also, we state important results like Brownawell-Maser theorem, Castelnuovo's inequality, and Riemann's inequality  which are essential to prove main theorems. Finally, in Section \ref{sec:proof}, we prove main theorems. Our proof closely follow the proofs in \cite{fuchs2021diophantine} and \cite{fuchs2019decomposable}.

	\section{Notation and results} \label{sec:notation}
	Let $K$ be a number field, $\mathcal{O}_S$ be the ring of integers of $K$ and $(U_n(x))_{n\geq 0}$ be the recurrence sequence of polynomials of order $d\geq2$ defined by
	\begin{equation}\label{maineq}
	U_n(x)=a_1(x)\alpha_1(x)^n+\cdots+a_d(x)\alpha_d(x)^n,
	\end{equation}
	where $a_1(x),\ldots,a_d(x)\in K[x]$ and the polynomial characteristic roots $\alpha_1(x),\ldots,\alpha_d(x)\in K[x]$. Again, we rewrite $a_1(x)\alpha_1(x)^n$ as follows:
    \begin{equation}\label{eqalpha}
        a_1(x)\alpha_1(x)^n=\eta_{\ell}x^{\ell}+\eta_{\ell-1}x^{\ell-1}+\eta_{\ell-2}x^{\ell-2}+\cdots+\eta_{0},
    \end{equation} where $\ell=n\deg \alpha_1+\deg a_1$ and $\eta_i\in K, 0\leq i\leq \ell$. Furthermore, we assume the following:
    \begin{itemize}
        \item[(i)] $U_n(x)$ satisfy the dominant root condition, that is $\deg \alpha_1>\max_{i=2,\ldots,d}\deg \alpha_i$ and having at most one constant characteristic root.
        \item[(ii)] $\deg a_1\geq \max_{i=2,\ldots,d}\deg a_i$.
        \item[(iii)] $U_n(x)$ cannot be written in the form $\hat{a_1}\hat{\alpha_1}(x)^n+\hat{a_2}\hat{\alpha_2}^n$ for $\hat{\alpha_1}(x)\in K[x]$ a perfect power of a linear polynomial and $\hat{a_1}, \hat{a_2},\hat{\alpha_2}\in K$.
        \item[(iv)] The coefficients $\eta_{\ell-1}$ and $\eta_{\ell-2}$ in \eqref{eqalpha} are zero.
    \end{itemize}
     We will refer to the assumptions (i)-(iv) by saying $	U_n(x)=a_1(x)\alpha_1(x)^n+\cdots+a_d(x)\alpha_d(x)^n$ {\it n-th polynomial in a linear recurrence sequence of the desired structure.}

     \begin{remark}
         It is possible to choose the coefficients of $a_1(x)$ so that the assumption (iv) holds. For that, let $n=3$, $a_1(x)=b_0+b_1x+b_2x^2$, $\alpha_1(x)=c_0+c_1x+c_2x^2$, where $b_i,c_i\in K, i=0,1,2, b_2,c_2\neq 0$. Then 
         \begin{align*}
             \alpha_1(x)^3=c_0^3&+3c_0^2c_1x+(3c_0c_1^2+3c_0^2c_2)x^2+(6c_0c_1c_2+c_1^3)x^3\\&+(3c_0c_2^2+3c_1^2c_0)x^4+3c_1c_2^2x^5+c_2^3x^6.
         \end{align*}
         Let 
         \begin{equation}\label{eqrem1}
             a_1(x)\alpha_1(x)^3=\sum_{i=0}^{8}\eta_ix^i.
         \end{equation} Observe that $\ell=8$. We now determine the conditions under which the coefficients $\eta_7$ and $\eta_6$ are zero. From \eqref{eqrem1} 
         \begin{equation*}
             \eta_7=b_1c_2^3+3b_2c_1c_2^2=0
         \end{equation*} implies 
         \begin{equation*}
             b_1=\frac{-3b_2c_1}{c_2}
         \end{equation*} and 
         \begin{equation*}
             \eta_6=b_0c_2^3+b_1(3c_1c_2^2)+b_2(3c_0c_2^2+3c_1^2c_2)=0 
         \end{equation*} implies 
         \begin{equation*}
             b_0=\frac{3b_2(2c_1^2-c_0c_2)}{c_2^2}. 
         \end{equation*} For simplification and verification, we use \textit{Mathematica} to compute the general formulas for the coefficients and determine the conditions under which they vanish. The code is written specifically for the given degrees of $a_1(x)$ and $\alpha_1(x)$, and the value of $n$.
 
     \end{remark}
	
	We say that a polynomial $f$ of degree at least 2 is decomposable if it can be written as $f=g\circ h$, where $g$ and $h$ are polynomials with degrees both at least 2. If no such representation exists, then $f$ is said to be indecomposable. Let $S$ be a finite set of places of $K$ containing all Archimedean places. We denote by $\mathcal{O}_S$ the ring of $S$-integers of the field $K$. We say that the equation $f(x)=g(y)$ has infinitely many solutions with bounded $\mathcal{O}_S$-denominator if there exists $\delta\in \mathcal{O}_S$ such that $f(x)=g(y)$ has infinitely many solutions $(x,y)\in K\times K$ with $\delta x,\delta y\in \mathcal{O}_S$.

    The first result of this paper is stated as follows.
	
\begin{theorem}\label{thm1}
Let $K$ be the number field, $S$ be a finite set of places of $K$ containing all Archimedean places, $(U_n(x))_{n\geq 0}$ and $(V_m(y))_{m\geq 0}$  be linear recurrence sequences of the desired structure with power sum representation  $U_n(x)=a_1(x)\alpha_1(x)^n+\cdots+a_d(x)\alpha_d(x)^n$ and $V_m(y)=b_1(y)\beta_1(y)^m+\cdots+b_s(y)\beta_s(y)^m$ respectively, where $a_i(x),$ $b_j(x),\alpha_i(x),\beta_j(x)\in K[x]$. If $n,m\geq 3$ and $U_n(x)$ is indecomposable, then 
\begin{equation}\label{thmeq}
			U_n(x)=V_m(y)
\end{equation} has infinitely many solutions $(x,y)\in K\times K$ with a bounded $\mathcal{O}_S$-denominator if and only if there exists a polynomial $Q(y)\in K[y]$ such that $V_m(y)=U_n(Q(y))$ holds identically.

Moreover, if $V_m(y)$ is also indecomposable, then in the above statement we can restrict $Q(y)$ to be linear.
\end{theorem}

Let $\C[x]$ be the polynomial in the indeterminate $x$. In the rest of the paper, we are interested in subsets of $\C[x]$, that is, $\{W_n(x): n\in \N\}$ that consist of elements of a simple non-degenerate third order linear recurrence sequence $(W_n(x))_{n\geq 0}$ given by
\begin{equation}\label{recrel}
  W_{n+3}(x)=a(x)W_{n+2}(x)+b(x)W_{n+1}(x)+c(x)W_{n}(x)\quad \text{for } n\geq 0,  
\end{equation}
with $a(x),b(x),c(x),W_0(x),W_1(x),W_2(x)\in \C[x]$.
We denote $\lambda_1(x),\lambda_2(x),\lambda_3(x)$ as the roots of the characteristic polynomial 
\[T^3-a(x)T^2-b(x)T-c(x).\] Setting $S=T-\frac{1}{3}a(x)$ the characteristic polynomial becomes \[S^3-p(x)S-q(x),\] where
\[
p(x)=\tfrac{1}{3}a(x)^2+b(x), \quad 
q(x)=\tfrac{2}{27}a(x)^3+\tfrac{1}{3}a(x)b(x)+c(x).
\]
Let
\begin{align*}
    \mathfrak{D}(x) =& \left( \frac{q(x)}{2} \right)^2 - \left( \frac{p(x)}{3} \right)^3\\
      =& \tfrac{1}{27}a(x)^3c(x)-\tfrac{1}{108}a(x)^2b(x)^2+\tfrac{1}{6}a(x)b(x)c(x)+\tfrac{1}{4}c(x)^2-\tfrac{1}{27}b(x)^3.
\end{align*}
Moreover, let
\[
u(x)=\sqrt[3]{\frac{q(x)}{2}+\sqrt{\mathfrak{D}(x)}}, \qquad 
v(x)=\sqrt[3]{\frac{q(x)}{2}-\sqrt{\mathfrak{D}(x)}}.
\]
Then we have by Cardano’s formulae
\begin{align}
\lambda_1(x) &= u(x)+v(x)+\tfrac{1}{3}a(x), \label{eqt2.6} \\
\lambda_2(x) &= -\frac{u(x)+v(x)}{2}+ i\sqrt{3}\,\frac{u(x)-v(x)}{2}+\tfrac{1}{3}a(x), \label{eqt2.7}\\
\lambda_3(x) &= -\frac{u(x)+v(x)}{2}- i\sqrt{3}\,\frac{u(x)-v(x)}{2}+\tfrac{1}{3}a(x). \label{eqt2.8}
\end{align}
Since $(W_n(x))_{n\geq 0}$ is simple, we have for $n\geq 0$
\begin{equation}\label{eeq1}
W_n(x)=w_1(x)\lambda_1(x)^n+w_2(x)\lambda_2(x)^n+w_3(x)\lambda_3(x)^n,  
\end{equation}
where
\[
w_1(x),w_2(x),w_3(x)\in K(i\sqrt{3})\left( x,\sqrt{\mathfrak{D}(x)},u(x),v(x)\right).
\]

 A polynomial $h$ is said to be \emph{cyclic} if 
    \[
    h(x) = \ell_1(x) \circ x^n \circ \ell_2(x)
    \]
    for some $n \geq 2$ and linear polynomials $\ell_1, \ell_2\in \mathbb{C}[x]$. A polynomial $h$ is said to be \emph{dihedral} if 
    \[
    h(x) = \ell_1(x) \circ T_n(x) \circ \ell_2(x)
    \]
    for some $n \geq 3$ and linear polynomials $\ell_1, \ell_2\in \C [x]$, where $T_n(x)$ is the $n$-th Chebyshev polynomial of the first kind, defined by the recurrence
\[
T_{n+2}(x) = 2xT_{n+1}(x) - T_n(x), \quad T_0(x) = 1, \; T_1(x) = x, \quad \text{for } n\in \N.
\]
It is well known that $T_{mn}=T_m\circ T_n$ for any $m,n\in \N$. Assume that $W_n$ is decomposable for some $n\in \N$ and write $W_n(x)=g(h(x))$, where $h$ is indecomposable, and thus $\deg h\geq 2$. By Gauss's lemma it follows that the polynomial $h(T)-h(x)\in \C(h(x))[T]$ is irreducible. Since $\deg h\geq 2$, there exists a root $y\neq x$ in its splitting field over $\C(h(x))$. Clearly, $h(x)=h(y)$.
Conjugating over $\C(h(x))$ via $x\mapsto y$, we get a sequence $(W_n(y))_{n\geq0}$ with $W_n(y)\in \C[y]$, which satisfies the same non-degenerate simple recurrence relation as $(W_n(x))_{n\geq 0}$ with $x$ replaced by $y$. We note that
	\begin{equation*}
		W_n(y)=t_1\mu_1^n+t_2\mu_2^n+t_3\mu_3^n,
	\end{equation*} where $t_1,t_2,t_3\in K(i\sqrt{3})(y,\sqrt{\mathfrak{D}(y)},u(y),v(y)) $ and $\mu_1,\mu_2,\mu_3$ are distinct roots of the characteristic polynomial of $(W_n(y))$ in its splitting field over $\C(y)$. Since $h(x)=h(y)$, we get $W_n(x)=W_n(y)$, that is
\begin{equation}\label{thm2eq1}
w_1\lambda_1^n+w_2\lambda_2^n+w_3\lambda_3^n=t_1\mu_1^n+t_2\mu_2^n+t_3\mu_3^n.
\end{equation}
To state our next result, we need the following assumption, that is,
\begin{equation}\label{eqval13}
        \nu\left(\frac{u}{v}\right)\neq 0\quad \text{with } \max \{\nu(u), \nu(v)\}>0,
    \end{equation} and
    \begin{equation}\label{eqval14}
        \nu(\mu_1)\nu(\mu_2)<0 \quad and \quad\nu\left(\frac{u+v}{a}\right)\neq 0 \quad\text{with } \max\{\nu(u+v),\nu(a)\}>0.
    \end{equation}

\begin{theorem}\label{2thm}
    Let $(W_n(x))_{n\geq0}$ be a sequence of polynomials defined as in \eqref{eeq1}. Assume that $3\deg a>\deg c, 2\deg a>\deg b$ and $\deg a+\deg c>2\deg b$ and for all valuations $\nu$, \eqref{eqval13} and \eqref{eqval14} hold. Then there exists a positive real constant $C=C(\{a,b,c,W_i: i=0,1,2\})$ such that if for some $n$ we have 
     \[W_n(x)=g(h(x)),\]
     where $h$ is indecomposable and neither cyclic nor dihedral, and if \eqref{thm2eq1} has no proper vanishing subsum, then $\deg g \leq C$.
\end{theorem}

For our final result, we consider a simple and non-degenerate binary recurrent sequence $(u_n(x))_{n\geq0}$ satisfying
\begin{equation}\label{eq2.21}
		u_{n+2}(x)=A_1(x)u_{n+1}(x)+A_0(x)u_n(x), \quad n\geq 0,
\end{equation} with $u_0(x),u_1(x),A_0(x),A_1(x)\in \C[x]$. By $\alpha_1(x), \alpha_2(x)$, we denote the roots of the corresponding characteristic polynomial
\[T^2-A_1(x)T-A_0(x)\]
and $D(x) = \sqrt{A_1(x)^2+4A_0(x)}$. Since the sequence $(u_n(x))_{n\geq 0}$ is simple, we have that $D(x)\neq 0$. Then for any $n\geq0$,
\begin{equation}\label{eq5.9}
		u_n(x)=\sigma_1\alpha_1^n+\sigma_2\alpha_2^n,
\end{equation}
where 
\begin{equation*}\label{eq23}
\alpha_1 = \frac{A_1(x) - \sqrt{A_1(x)^2 + 4A_0(x)}}{2}, \qquad
\alpha_2 = \frac{A_1(x) + \sqrt{A_1(x)^2 + 4A_0(x)}}{2}
\end{equation*} and 
\begin{equation*}\label{eq23a}
    \sigma_1= \frac{u_{1}(x)-u_{0}(x)\alpha_{2}}{\alpha_{1}-\alpha_{2}}, \qquad  \sigma_2= \frac{u_{0}(x)\alpha_1-u_{1}(x)}{\alpha_{1}-\alpha_{2}}.
\end{equation*} 
Assume that $u_{n_1}+u_{n_2}$ is decomposable for some $n_1,n_2\in \N$ and write $u_{n_1}(x)+u_{n_2}(x)=g(h(x))$, where $h$ is indecomposable, and thus $\deg h\geq 2$. By arguing as in the third order recurrence sequence, we have that
\begin{equation*}
u_n(x)=\sigma_1\alpha_1^n+\sigma_2\alpha_2^n,
\end{equation*} and 
\begin{equation*}
u_n(y)=\delta_1\beta_1^n+\delta_2\beta_2^n,
\end{equation*}
where $\alpha_1, \alpha_2$ are the distinct roots of the characteristic polynomial of $(u_n(x))$ in its splitting field $L_1/\C(x)$, and $\sigma_1,\sigma_2 \in L_1$ and $\beta_1, \beta_2$ are distinct roots of the characteristic polynomial of $(u_n(y))$ in its splitting field $L_2/\C(y)$, and $\delta_1, \delta_2\in L_2$. Since $h(x)=h(y)$, we get $u_{n_1}(x)+u_{n_2}(x)=u_{n_1}(y)+u_{n_2}(y)$, that is
\begin{equation}\label{eq5.8}
\sigma_1\alpha_1^{n_1}+\sigma_2\alpha_2^{n_1}+\sigma_1\alpha_1^{n_2}+\sigma_2\alpha_2^{n_2}=\delta_1\beta_1^{n_1}+\delta_2\beta_2^{n_1}+\delta_1\beta_1^{n_2}+\delta_2\beta_2^{n_2}.
	\end{equation}
Our third result is as follows.
\begin{theorem}\label{thm2}
	 Let $(u_n(x))_{n\geq 0}$ be a sequence of polynomials defined as in \eqref{eq5.9}. There exists a positive real constant $C=C(\{a_i, u_i: i=0,1\})$ such that if for some $n_1\leq  n_2$ we have 
     \[u_{n_1}(x)+u_{n_2}(x)=g(h(x)),\]
     where $h$ is indecomposable and neither cyclic nor dihedral, and if \eqref{eq5.8} has no proper vanishing subsum, then $\deg g \leq C$.
	\end{theorem}
In the following remarks we show that if $h$ is cyclic or dihedral, then the $\deg g$ is not bounded.
    \begin{remark}
        For $m,n,k \in \mathbb{Z}_{\geq 0}$, we have 
        \begin{equation}\label{remeq1}
            T_{nk}(x)\,T_{mk}(x) \;=\; \frac{T_{nk+mk}(x)+T_{nk-mk}(x)}{2}.
        \end{equation}
Using the identity $T_{mn}(x)=T_m \circ T_n(x)$, we rewrite \eqref{remeq1} as  
\[
T_{nk+mk}(x)+T_{nk-mk}(x)=2\,(T_n \circ T_k)(x)\,(T_m \circ T_k)(x) 
\;=\; 2\,(T_n T_m)\circ T_k(x)  .
\]
It follows that the degree of $T_k(x)$ necessarily depends on $k$. Therefore, if the polynomial $h(x)$ in Theorem~\ref{thm2} is of dihedral type, then the degree of $g$ cannot be bounded.
    \end{remark}

    \begin{remark}
        Let $u_n(x)=x^n+1$. Then, for any $m,n,k \in \mathbb{Z}_{\geq 0}$, we have  
\[
u_{mk}(x)+u_{nk}(x) = x^{mk}+1+x^{nk}+1 = (x^m+x^n+2)\circ x^k.
\]  
Note that $x^k$ is a cyclic polynomial whose degree varies with $k$. Hence, the assumption that $h(x)$ is not cyclic in Theorem \ref{thm2} is indeed necessary.
    \end{remark}

The following is an application to the Theorem \ref{thm2}. For that we consider another simple and non-degenerate sequence $(v_n(x))_{n\geq0}$ defined by
\begin{equation}\label{eq2.22}
		v_{n+2}(x)=B_1(x)v_{n+1}(x)+B_0(x)v_n(x), \quad n\geq 0,
\end{equation} where $v_0(x),v_1(x),B_0(x),B_1(x)\in \C[x]$.
	\begin{corollary}\label{corolary1}
	 Consider sequences $(u_n(x))_{n\geq 0}$ and $(v_n(x))_{n\geq 0}$ satisfying \eqref{eq2.21} and \eqref{eq2.22} respectively with $\deg u_n \geq \deg v_m$ for any $n,m\in \Z$. Then there exists a constant $C=C(\{A_i,u_i : i=0,1\})>0$ with the following property. If $(u_n(x))_{n\geq0}$, $(v_n(x))_{n\geq 0}$,  $u_{n_1}(x)+u_{n_2}(x)$, and $v_{m_1}(x)+v_{m_2}(x)$ are not composites of cyclic or dihedral polynomials for any $n_1,n_2,m_1,m_2 \geq 0$, and the equation 
     \begin{equation*}
         u_{n_1}(x)+u_{n_2}(x)=v_{m_1}(y)+v_{m_2}(y)
     \end{equation*} has infinitely many integer solutions $x,y$ and does not have proper vanishing subsums, then \[u_{n_1}(x)+u_{n_2}(x)=(v_{m_1}+v_{m_2})(\ell(x))\] for a linear $\ell(x)\in \C[x]$.
	\end{corollary}
    \section{Auxiliary results} \label{sec:auxiliary}
    In this section,we collect some important results which we will need in our proofs. First we outline the finiteness criterion proposed by Bilu and Tichy in \cite{bilu2000diophantine}, for which we define the five types of standard pairs $(f(x),g(x))$.
	 
	Let $K$ be a number field and $a,b\in K^{*}, m,n\in \N,\,  p(x)\in K[x]$ be a non zero polynomial (which can be constant) and $D_n(x,r)$ be the $n$-th Dickson polynomial with parameter $r$ given by
	\[D_n(x,r)=\sum_{j=0}^{\left\lfloor \frac{n}{2}\right\rfloor}\frac{n}{n-j}\binom{n-j}{j}(-r)^jx^{n-2j}.\]
	Here $\left\lfloor \frac{n}{2}\right\rfloor$ denotes the largest integer $\leq \frac{n}{2}$.

    Using this notation we have the following five kinds of standard pairs (over $K$); in each of them the two coordinates can be switched. Standard pairs of polynomials over $K$ are listed in the following table.
    \begin{table}[h!]
    \centering
    \begin{tabular}{|l l l|}
    \hline
    \textbf{Kind} &
    \textbf{Standard pair (or switched)} &
    \textbf{Parameter restrictions} \\
    \hline

    First  &
    $\bigl(x^{m},\, a x^{n} p(x)^{m}\bigr)$ &
    $0\leq n < m,\ \gcd(n,m)=1,\ n+\deg p > 0$ \\[2mm]

    Second &
    $\bigl(x^{2},\, (a x^{2} + b)p(x)^{2}\bigr)$ &
    -- \\[2mm]

    Third  &
    $\bigl(D_{m}(x,a^{n}),\, D_{n}(x,a^{m})\bigr)$ &
    $\gcd(m,n)=1$ \\[2mm]

    Fourth &
    $\left(a^\frac{-m}{2} D_{m}(x,a),\, -b^\frac{-n}{2} D_{n}(x,b)\right)$ &
    $\gcd(m,n)=2$ \\[2mm]

    Fifth  &
    $\bigl((a x^{2}-1)^{3},\, 3x^{4}-4 x^{3}\bigr)$ &
    -- \\ \hline
    \end{tabular}
    \end{table}
   
    \begin{table}[h!]
    \centering
    \begin{tabular}{|l l|}
    \hline
    \textbf{}  \textbf{Specific pair (or switched)} &
    \textbf{Parameter restrictions} \\
    \hline
    $(D_m(x, r^{n/d}), -D_n(x\cos(\pi/d), r^{m/d}))$ &
    $d= \gcd(m,n) \geq 3, n, \cos(2\pi/d) \in K$ \\[2mm]
    \hline
    \end{tabular}
    \end{table}

Note that $-D_n(x\cos(\pi/d), r^{m/d}) \in K[x]$. The key result we use to prove Theorem \ref{thm1} is the following theorem due to Bilu and Tichy \cite[Theorem 10.5]{bilu2000diophantine}.
	\begin{theorem}\label{biluthm}
		Let $K$ be the number field, $S$ be a finite set of places of $K$ containing all Archimedean places, and $f(x),g(x)\in K[x]$. Then the following are equivalent.
		\begin{itemize}
			\item[(i)] The equation $f(x) = g(y)$ has infinitely many solutions with a bounded $\mathcal{O}_S$-denominator.
			\item[(ii)] We have $f=\phi \circ f_1\circ \lambda$ and $g=\phi\circ g_1\circ \mu$, where $\lambda(x),\mu(x)\in K[x]$ are linear polynomials, $\phi(x)\in K[x]$, and $(f_1(x),g_1(x))$ is a standard or specific pair over $K$ such that the equation $f_1(x) = g_1(y)$ has infinitely many solutions with a bounded $\mathcal{O}_S$-denominator.
		\end{itemize}
        (In particular if $K=\Q,$ then the word ``specific'' in $(ii)$ may be skipped.)
	\end{theorem}

      Let $L$ be a finite algebraic extension of function field $\C(x)$. For $a\in \C$ define the valuation $\nu_a$ as follows. For $q(x)\in \C(x)$ let $q(x)=(x-a)^{\nu_a(q)}A(x)/B(x)$, where $A,B\in \C[x]$ and $A(a)B(a)\neq 0$. Let $\nu_{\infty}$ be the infinite valuation which is defined by $\nu_{\infty}(Q):=\deg B-\deg A$ for $Q(x)=A(x)/B(x)$, where $A,B\in \C[x]$. These are all (normalized) discrete valuations on $\C(x)$. All of them can be extended in at most $[L:\C(x)]$ ways to a discrete valutaion on $L$ and again in this way one obtains all discrete valuations on $L$. Furthermore, for $f\in L^{*}$ the sum formula $\sum \nu(f)=0$ holds, where the sum is taken over all discrete valuations on $L$.  We just mention that there are different equivalent descriptions of the notion of discrete valuations as places or the rational points on any non-singular complete curve over $\C$ with function field $L$. 
      
      Now define the {\em projective height} $\mathcal{H}$ of $u_1, \ldots, u_n \in L/\C, n\geq 2$ and not all $u_i$ are zero, is defined by
	\begin{equation*}\label{eq1_1}
		\mathcal{H}(u_1, \ldots, u_n): = -\sum_{\nu} \min \{\nu(u_1), \ldots, \nu(u_n)\},
	\end{equation*}
	where $\nu$ runs over all places of $L/\C(x)$. For a single element $f\in L^{*}$, we set
	\begin{equation*}
		\mathcal{H}(f) :=\mathcal{H}(1, f)= -\sum_{\nu} \min \{0, \nu(f)\},
	\end{equation*}
	where the sum is taken over all discrete valuations on $L$; thus for $f\in \C(x)$ the height $\mathcal{H}(f)$ is the number of poles of $f$ counted according to multiplicity. We note that if $f\in \C[x]$, then $\mathcal{H}(f)  = [L:\C(x)]\deg f$.  Also, observe that $\nu(f) \neq 0$ only for a finite number of valuations $\nu$ and hence, by the sum formula $\mathcal{H}(f) = \sum_{\nu} \max (0, \nu(f))$. For $f = 0$, we define $\mathcal{H}(f) = \infty$. This height function satisfies some basic properties which we listed below.
	
	\begin{lemma}\label{lem1}
		Suppose that $\mathcal{H}$ is the projective height on $L/\C(x)$ defined as above. Then for $f, g\in L^{*}$ the following properties hold:
		\begin{enumerate}[label=(\alph*)]
			\item $\mathcal{H}(f) \geq 0$ and $\mathcal{H}(f) = \mathcal{H}(1/f)$,
			\item $\mathcal{H}(f) - \mathcal{H}(g) \leq \mathcal{H}(f+g)\leq \mathcal{H}(f) +\mathcal{H}(g)$,
			\item $\mathcal{H}(f) -\mathcal{H}(g) \leq \mathcal{H}(fg)\leq \mathcal{H}(f) +\mathcal{H}(g)$,
			\item $\mathcal{H}(f^n) = |n|\cdot \mathcal{H}(f)$,
			\item $\mathcal{H}(f)=0\iff f\in \C^{*}$,
			\item $\mathcal{H}(P(f)) = \deg P\cdot \mathcal{H}(f)$ for any $P\in \C[T]\setminus \{0\}$.         
		\end{enumerate}
	\end{lemma}
	\begin{proof}
		For a detailed proof, see \cite[Lemma 2]{fuchs2019decomposable}.
	\end{proof}

    In the following lemma, we establish that the right hand side inequality of $(b)$ and $(c)$ of Lemma \ref{lem1} becomes equality under suitable additional conditions.
    \begin{lemma}\label{lemequality}
    Let $f,g\in L^{*}$. Then
        \begin{itemize}
            \item[(i)] $\mathcal{H}(f+g)=\mathcal{H}(f)+\mathcal{H}(g)$, if $\nu(f)\neq \nu (g)$ and $\max \{\nu(f),\nu (g)\}>0$ for all valuations $\nu$ over $L$.
            \item[(ii)] $\mathcal{H}(fg)=\mathcal{H}(f)+\mathcal{H}(g)$, if $\nu(f)\nu (g)>0$ for all valuations $\nu$ over $L$. 
        \end{itemize}
    \end{lemma}
    \begin{proof}
        \begin{enumerate}
            \item[($i$)] Recall that
            \[\mathcal{H}(f+g)=- \sum_{\nu} \min \{ 0,\nu (f+g)\}.\]
            Therefore, it suffices to verify that for every valuation $\nu$, 
            \begin{equation}\label{eqnew1}
                \min \{ 0,\nu (f+g)\}=\min \{ 0,\nu (f)\}+\min \{ 0,\nu (g)\}.
            \end{equation} Since 
            \begin{equation*}
                \nu(f+g)\geq \min\{\nu(f),\nu(g)\}.
            \end{equation*}and because we assume $\nu(f)\neq \nu(g)$, we infact have 
            \begin{equation*}
                \nu(f+g)= \min\{\nu(f),\nu(g)\}.
            \end{equation*} Without loss of generality assume that $\nu(f)<\nu(g)$; hence $\nu(f+g)=\nu(f)$. Fix a valuation $\nu$. Now we consider the following three cases.\\
            \textit{Case 1:} Suppose  $\nu(f+g)=0.$ Then $\nu(f)=0$ and $\nu(g)>0$. Thus
            \[
            \min \{0, \nu(f+g)\}=0=\min \{0, \nu(f)\}+\min \{0, \nu(g)\}.
            \]  So \eqref{eqnew1} holds.\\
            \textit{Case 2:} Suppose $\nu(f+g)>0$. Then $\nu(f)>0$ and since $\nu(f)<\nu(g)$, also $\nu(g)>0$. Hence, in \eqref{eqnew1} $\nu(f+g),\nu(f),\nu(g)$ are positive, and therefore is zero and thus equality holds.\\
            \textit{Case 3:} Suppose $\nu(f+g)<0$. Then $\nu(f)<0$. Also note that either $\nu(g)<0$ or $\nu(g)>0$. Since by our assumption $\max \{\nu(f), \nu(g)\}=\nu(g)>0$ and therefore $\min \{0,\nu (g)\}=0$. Thus,
            \[
            \min\{0,\nu(f+g)\}=\nu(f)=\min\{0,\nu(f)\}+\min\{0,\nu(g)\}
            \]and\eqref{eqnew1} holds. This completes the proof of part $(i)$.
            \item[($ii$)] We have $\nu(fg)=\nu(f)+\nu(g)$. Therefore,
            \begin{equation*}
                \mathcal{H}(fg)=-\sum_{\nu} \min \{0,\nu(fg)\}= -\sum_{\nu} \min \{0,\nu(f)+\nu(g)\}. 
            \end{equation*} So it is suffices to verify that for every valuation $\nu$,
            \begin{equation}\label{eqnew2}
                \min \{ 0,\nu (f)+\nu(g)\}=\min \{ 0,\nu (f)\}+\min \{ 0,\nu (g)\}.
            \end{equation} We check this by considering all possible sign combination of $\nu(f)$ and $\nu(g)$.\\
            \textit{Case 1:} Suppose $\nu(f)=\nu(g)=0$. Then, $\nu(f)+\nu(g)=0$, and the identity \eqref{eqnew2} holds trivially. \\
            \textit{Case 2:} Suppose $\nu(f)>0$ and $\nu(g)=0$. Then $\nu(f)+\nu(g)=\nu(f)$ and hence all the minimum in $\eqref{eqnew2}$ are zero and the identity holds.\\
            \textit{Case 3:} Suppose $\nu(f)>0$ and $\nu(g)>0$. Then 
            \[
            \min \{0, \nu(f+g)\}=0=\min \{0, \nu(f)\}+\min \{0, \nu(g)\}.
            \]
            \textit{Case 4:} Suppose $\nu(f)<0$ and $\nu(g)<0$. Then $\nu(f)+\nu(g)<0$. Thus, 
            \begin{equation*}
                \min \{ 0,\nu(f)+\nu(g)\}= \nu(f)+\nu(g)= \min \{ 0,\nu(f)\}+\min \{0,\nu(g)\}.
            \end{equation*}
             Similarly, if we interchange the roles of $f$ and $g$ in all the above cases, we get the same conclusion. This completes the proof of Lemma \ref{lemequality}.
        \end{enumerate} 
    \end{proof}
	
	The following theorem is due to Brownawell and Masser \cite{brow1986} and this is an immediate consequence of \cite[Theorem B, Corollary 1]{Fuchs2012}. This result gives an upper bound for the height of $S$-units, which arise as a solution of certain $S$-unit equations.
	\begin{proposition}[Brownawell-Masser]\label{browthm}
		Let $L/\C(x)$ be a function field of one variable of genus $\mathfrak{g}$. Moreover, for a finite set $S$ of discrete valuations, let $u_1, \ldots, u_n$ be not all constant $S$-units and 
		\[1+u_1+\cdots +u_n= 0,\] where no proper subsum of the left hand side vanishes. Then we have 
		\begin{equation*}\label{bmineq}
			\max_{i=1, \ldots, n} \mathcal{H}(u_i) \leq \binom{n}{2}(|S|+\max(0, 2\mathfrak{g}-2)). 
		\end{equation*}
	\end{proposition}

In addition, we employ the following classical estimates for the genus of a compositum of function fields.
\begin{proposition}[Castelnuovo's Inequality]\label{castelnuvothm}
     Let $F/\mathbb{C}$ be a function field of one variable of genus $\mathfrak{g}$. Suppose there are given two subfields $F_1/\mathbb{C}$ and $F_2/\mathbb{C}$ of $F/\mathbb{C}$ satisfying  

     \begin{enumerate}
     \item[(1)] $F = F_1F_2$ is the compositum of $F_1$ and $F_2$.
     \item[(2)] $[F:F_i] = k_i$, and $F_i/\mathbb{C}$ has genus $\mathfrak{g}_i$ $(i=1,2)$.
     \end{enumerate}

     Then we have
     \[
      \mathfrak{g} \leq k_1 \mathfrak{g}_1 + k_2 \mathfrak{g}_2 + (k_1 - 1)(k_2 - 1).
    \]
\end{proposition}
\begin{proof}
    See \cite[pp.~130]{stichtenoth2009algebraic}.
\end{proof}

\begin{proposition}[Riemann's Inequality]\label{nthm5}
    Suppose that 
$F = \mathbb{C}(x,y)$. Then we have the following estimate for the genus $\mathfrak{g}$ of $F/\mathbb{C}$:

\[
\mathfrak{g} \leq ([F:\mathbb{C}(x)] - 1)([F:\mathbb{C}(y)] - 1).
\]
\end{proposition}
\begin{proof}
    See \cite[pp.~132]{stichtenoth2009algebraic}.
\end{proof}

 We make use of the following results to prove Theorem \ref{2thm} and Theorem \ref{thm2}.

    \begin{proposition}\label{lem3.3}
Let $h \in \mathbb{C}[x]$ be indecomposable and let $y \neq x$ be a root of 
$h(X) - h(x) \in \mathbb{C}(x)[X]$. Then either 
$\mathbb{C}(x) \cap \mathbb{C}(y) = \mathbb{C}(x)$ and $h$ is cyclic or 
$\mathbb{C}(x) \cap \mathbb{C}(y) = \mathbb{C}(h(x))$.
\end{proposition}
\begin{proof}
    For a detailed proof see \cite[Lemma 4]{fuchs2019decomposable} 
\end{proof}

\begin{proposition}\label{nlemma1}
Let $h \in \mathbb{C}[x]$ be indecomposable and let $y \neq x$ be a root of 
$h(X) - h(x) \in \mathbb{C}(x)[X]$. Then the following hold.

\begin{enumerate}
    \item[(1)] For $q \in \mathbb{C}[h(x)]$ we have $q(x) = q(y)$. Furthermore, if $h$ is not cyclic and $q(x) = q(y)$ for some $q \in \mathbb{C}[x]$, then $q \in \mathbb{C}[h(x)]$.
    
    \item[(2)] Let $d := [\mathbb{C}(x,y):\mathbb{C}(x)]$. Then $d \leq \deg h - 1$.
    
    \item[(3)] The genus of the function field $\mathbb{C}(x,y)$ (over $\mathbb{C}$) is not greater than $\dfrac{(d-1)(d-2)}{2}$.
\end{enumerate}
\end{proposition}

\begin{proof}
    See \cite[Lemma 5]{fuchs2019decomposable} 
\end{proof}

\section{Proof of main results}	 \label{sec:proof}
 \subsection{Proof of Theorem \ref{thm1}}
	We observe by the dominant root condition that $\deg \alpha_1\geq 1$ and $\deg \beta_1\geq 1$. Also note that, we can neither have $\deg \alpha_1=1$ nor $\deg\beta_1=1$. Otherwise, if $\deg \alpha_1=1$, then $U_n(x)$ would possess exactly two characteristic roots, one being a constant. However, this form is excluded by the assumption. A similar argument applies to $\deg\beta_1$.
    So $\deg U_n=\deg a_1+n \deg \alpha_1\geq 6$. Similarly, the bound $\deg V_m=\deg b_1+m \deg\beta_1\geq 6$ holds. 
	
	Now assume that \eqref{thmeq} has infinitely many solutions $(x,y)\in K\times K$ with a bounded $\mathcal{O}_S$-denominator. Then by Theorem \ref{biluthm}, we have
	\begin{equation}\label{compoeq}
		U_n(x)=\phi(u(\lambda(x))),\quad V_m(x)=\phi(v(\mu(x)))
	\end{equation} where $\lambda(x), \mu(x)\in K[x]$ are linear polynomials, $(u(x),v(x))$ is a standard pair or a specific pair and $\phi(x)\in K[x]$.
	
	Suppose that $\deg \phi=1$. Let $\phi(x)=ax+b, ~a,b\in K$. Then \eqref{compoeq} implies that 
	\begin{equation*}
		U_n(x)=au(\lambda(x))+b,\quad V_m(x)=av(\mu(x))+b.
	\end{equation*} First assume that $(u(x),v(x))$ be a standard pair of the first kind. Then we have 
    \begin{equation*}
        U_n(x)=a(\lambda(x))^{k}+b \quad \text{or} \quad V_m(y)=a(\mu(y))^{\ell}+b 
    \end{equation*}
	 where $k=\deg a_1+n \deg \alpha_1$ and $\ell=\deg b_1+m \deg \beta_1$. But this is incompatible with the structure of $U_n(x)$ and $V_m(y)$. Thus $(u(x),v(x))$ cannot be a standard pair of the first kind.
	
	Since $\deg U_n\geq 3$ and $\deg V_m\geq 3$, $(u(x),v(x))$ cannot be a standard pair of second kind.
	
	If $(u(x),v(x))$ is a standard pair of the third kind, then we obtain
	\begin{equation}\label{dickscomp}
		U_n(x)=aD_p(\lambda(x),r)+b,
	\end{equation} with parameter $r$ and $\lambda(x)=cx+d, ~0\neq c,d\in K$.
	Since $U_n(x)$ is indecomposable and Dickson polynomials have the property that
	\[D_{kl}(x,r)=D_k(D_l(x,r),r^l)\] the index $p$ in \eqref{dickscomp} must be a prime. We have
	$D_p(x,r)=d_px^p+d_{p-2}x^{p-2}+\cdots ,$ where $$d_{p-2j}=\frac{p}{p-j}\binom{p-j}{j}(-r)^j,\quad (j=0,\ldots,\left\lfloor \frac{p}{2} \right\rfloor)$$ and 
    \begin{align*}
        \lambda(x)^p&=(cx+d)^p=c^px^{p}+ \binom{p}{1}c^{p-1}d x^{p-1}+\binom{p}{2}c^{p-2}d^2x^{p-2}+\cdots+d^p\\
        \lambda(x)^{p-2}&=(cx+d)^{p-2}=c^{p-2}x^{p-2}+ \binom{p-2}{1}c^{p-3}d x^{p-3}+\binom{p-2}{2}c^{p-4}d^2x^{p-4}\\
        &\hspace{4cm}+\cdots+d^{p-2},
    \end{align*}
and so on. Therefore, 
    \begin{align*}
       D_p(\lambda(x),r)&=d_p\lambda(x)^p+d_{p-2}\lambda(x)^{p-2}+\cdots ,\\
       &=d_pc^px^p+d_p\binom{p}{1}c^{p-1}dx^{p-1}+\left(d_{p-2}c^{p-2}+d_p\binom{p}{2}c^{p-2}d^2\right)x^{p-2}+\cdots .
    \end{align*}From \eqref{maineq} and \eqref{dickscomp}, we have
	\begin{align}\label{un=dp}
&a_1(x)\alpha_1(x)^n+\cdots+a_d(x)\alpha_d(x)^n\\
&=a\left(d_pc^px^p+d_p\binom{p}{1}c^{p-1}dx^{p-1}+(d_{p-2}c^{p-2}+d_p\binom{p}{2}c^{p-2}d^2)x^{p-2}+\cdots\right)+b \nonumber.
	\end{align}
 Comparing the degrees on both sides of \eqref{un=dp}
 \[\deg a_1+n\deg \alpha_1=p.\] Also we have $\deg a_i\leq \deg a_1$ and $\deg \alpha_i\leq \deg \alpha_1-1$ for all $i=2,\ldots,d$. Thus, for any $i=2,\ldots,d$ the following holds 
	\begin{align*}
		\deg a_i+n\deg \alpha_i&\leq \deg a_1+n(\deg \alpha_1-1)=\deg a_1+n\deg \alpha_1-n\\
		&=\deg a_1+p-\deg a_1-n = p-n  \leq p-3 \quad (\text{since }n\geq 3)\nonumber
	\end{align*} By assumption (iv) in theorem, the coefficients of $x^{p-2}$ and $x^{p-1}$ on the left hand side of \eqref{un=dp} are zero, where as if $d\neq 0$, then the coefficient of $x^{p-1}$ in right hand side is non zero, and if $d=0$, then the coefficient of $x^{p-2}$ in the right hand side is $d_{p-2}c^{p-2}$ which is non zero. Therefore, $(u(x),v(x))$ cannot be a standard pair of third kind.
	
	If $(u(x),v(x))$ is a standard pair of the fourth kind, then we have 
	\[U_n(x)=aD_k(\lambda(x),r)+b,\] where $k$ is even. This contradicts to the assumption that $U_n(x)$ is indecomposable.
	
	If $(u(x),v(x))$ is a standard pair of the fifth kind, then we have $u(x)=3x^4-4x^3$ or $v(x)=3x^4-4x^3$.  That is, $\deg a_1+n\deg \alpha_1=\deg U_n=4$ or  $\deg b_1+m\deg \beta_1=\deg V_m=4$. Since $\deg U_n \geq 6$ and $\deg V_m \geq 6$, this case is impossible. Hence, $(u(x),v(x))$ cannot be a standard pair of the fifth kind.

    If $(u(x),v(x))$ is a  specific pair, then
    \begin{equation*}
        u(x)=D_k(x,r^{\ell/d}) \quad \text{and}\quad v(x)=-D_{\ell}\left(x\cos \frac{\pi}{d}, r^{k/d}\right), \quad (\text{ or  switched})
    \end{equation*} where $d=\gcd(k,\ell)\geq 3$. Since $\gcd(k,\frac{\ell}{d})=1,$ we can express \eqref{compoeq} in the same form as \eqref{dickscomp}, and the subsequent arguments follow similarly. Therefore, $(u(x),v(x))$ cannot be a specific pair.
    
	Therefore, we conclude that $\deg \phi=1$ is impossible. So we have $\deg \phi>1$. Since $U_n$ is indecomposable, it follows that $\deg u=1$. Let $u(x)=c_1x+e_1$, $0\neq c_1, e_1\in K$. As a result, the following identities hold
    
	\[U_n(x)=\phi(c_1x+e_1) \quad \text{and} \quad V_m(y)=\phi(q(y)),\] where $q(y)\in K[y]$. Now set 
	\[Q(y):=\frac{q(y)-e_1}{c_1},\] which defines a polynomial in $K[y]$. This leads to the final identity
	\[U_n(Q(y))=\phi(q(y))=V_m(y).\]

    If $V_m(y)$ is indecomposable, then $q(y)$ is linear. Thus, by construction $Q(y)$ is linear, too.
	
	Conversely, if there exists a polynomial $Q(y)\in K[y]$ such that the identity $U_n(Q(y))=V_m(y)$, then \eqref{thmeq} has infinitely many rational solutions with a bounded $\mathcal{O}_S$-denominator. This completes the proof of Theorem \ref{thm1}. \qed
\subsection{Proof of Theorem \ref{2thm}}
    Firstly observe that  
    \begin{equation}\label{eqndeg}
        \deg q=3\deg a \quad \text{and} \quad \deg \mathfrak{D}=3\deg a+\deg c.
    \end{equation}
    Assume that $W_n(x)=g(h(x))$, where $h$ is indecomposable and neither cyclic nor dihedral. Recall that $x$ and $y$, which define \eqref{thm2eq1}, are such that $h(x)=h(y)$ and $x\neq y$. From Lemma \ref{lem3.3} it follows that $\C(x)\cap \C(y)=\C(h(x))$. Assume further that there is no proper vanishing subsum of \eqref{thm2eq1} and write it as 
    \begin{equation}\label{sumeq}
        1+ \frac{t_1\mu_1^n}{t_3\mu_3^n}+\frac{t_2\mu_2^n}{t_3\mu_3^n}-\frac{w_1\lambda_1^n}{t_3\mu_3^n}-\frac{w_2\lambda_2^n}{t_3\mu_3^n}-\frac{w_3\lambda_3^n}{t_3\mu_3^n}=0.
    \end{equation}
     Define
     \begin{equation*}
         p_1= \frac{t_1\mu_1^n}{t_3\mu_3^n},~p_2=\frac{t_2\mu_2^n}{t_3\mu_3^n}, ~p_3=-\frac{w_1\lambda_1^n}{t_3\mu_3^n}, ~p_4=-\frac{w_2\lambda_2^n}{t_3\mu_3^n}, ~p_5=-\frac{w_3\lambda_3^n}{t_3\mu_3^n}, 
     \end{equation*}
    and set
    \begin{equation*}
        q_1=\frac{\mu_1}{\mu_3},~ q_2=\frac{\mu_2}{\mu_3}, ~q_3=\frac{\lambda_1}{\mu_3}, ~q_4=\frac{\lambda_2}{\mu_3}, ~q_5=\frac{\lambda_3}{\mu_3},
    \end{equation*}
    \begin{equation*}
        r_1=\frac{t_1}{t_3}, ~r_2=\frac{t_2}{t_3}, ~r_3=\frac{w_1}{t_3}, ~r_4=\frac{w_2}{t_3}, ~r_5=\frac{w_3}{t_3}.
    \end{equation*}
    Consider the field $F=\C(x,y,\lambda_1,\lambda_2,\lambda_3, \mu_1,\mu_2,\mu_3)$ and define the projective height $\mathcal{H}$ on $F/\C(x)$ and by Lemma \ref{lem1} we obtain the estimate
    \begin{align*}
        \mathcal{H}\left(\frac{t_i\mu_i^n}{t_3\mu_3^n}\right)&\geq ~\mathcal{H}\left( \frac{\mu_i^{n}}{\mu_3^{n}}\right)-\mathcal{H}\left( \frac{t_i}{t_3}\right) \\&= n ~\cdot ~\mathcal{H}\left( \frac{\mu_i}{\mu_3}\right)-\mathcal{H}\left( \frac{t_i}{t_3}\right), \quad i=1,2,
    \end{align*}
    and similarly we get
    \begin{equation*}
        \mathcal{H}\left( \frac{w_i\lambda_i^n}{t_3\mu_3^n}\right)\geq n ~\cdot ~\mathcal{H}\left( \frac{\lambda_i}{\mu_3}\right)-\mathcal{H}\left( \frac{w_i}{t_3}\right), \quad i=1,2,3.
    \end{equation*} Thus we have, 
    \begin{equation*}
        \mathcal{H}(p_i)\geq n \mathcal{H}(q_i)- \mathcal{H}(r_i), \quad 1\leq i\leq 5.
    \end{equation*}
    Since $(W_n)_{n\geq 0}$ is non-degenerate $\mu_i/\mu_j$ is not a constant for any $i\neq j$, $\mathcal{H}(q_i)\neq 0$, for $i=1,2$. In particular, we have 
    \begin{equation}\label{eqt2.4}
        n\leq \left( \mathcal{H}(p_1)+\mathcal{H}(r_1)\right)\mathcal{H}(q_1)^{-1}.
    \end{equation}
    Moreover, the following upper bound can be deduced for the height of $W_n(x)$:
    \begin{align*}
        \mathcal{H}(W_n(x))=&\mathcal{H}(w_1\lambda_1^n+w_2\lambda_2^n+w_3\lambda_3^n)\\
        \leq & \mathcal{H}(w_1)+\mathcal{H}(w_2)+\mathcal{H}(w_3)+ n(\mathcal{H}(\lambda_1)+\mathcal{H}(\lambda_2)+\mathcal{H}(\lambda_3))\\
        \leq &n \left( \mathcal{H}(w_1)+\mathcal{H}(w_2)+\mathcal{H}(w_3)+\mathcal{H}(\lambda_1)+\mathcal{H}(\lambda_2)+\mathcal{H}(\lambda_3) \right).
    \end{align*}
    Using \eqref{eqt2.4}, we conclude that
    \begin{align}\label{eqt2.5}
        \mathcal{H}(W_n(x))\leq \left( \mathcal{H}(p_1)+\mathcal{H}(r_1)\right)&\mathcal{H}(q_1)^{-1}\left(\mathcal{H}(w_1)+\mathcal{H}(w_2)+\mathcal{H}(w_3)\right.\\
        &+\left.\mathcal{H}(\lambda_1)+\mathcal{H}(\lambda_2)+\mathcal{H}(\lambda_3)\right). \notag
    \end{align}
    Now consider \eqref{thm2eq1}, which by assumption has no proper vanishing subsum. Let $A=\{w_i, t_i, \lambda_i, \mu_i, i=1,2,3\}$ and put
        \[S:=\{\nu \in S_0: \nu (f)\neq 0 \text{ for some } f\in A\}\cup S_{\infty},\]
        where $S_0$ denotes the set of finite valuations and $S_{\infty}$ denotes the set of infinite valutaions on $F$. Then applying Proposition \ref{browthm} to \eqref{sumeq}, it follows that
        \begin{equation*}
          \mathcal{H}(p_1) \leq \max_{i=1, \ldots, 5} \mathcal{H}(p_i) \leq \binom{5}{2}\cdot \left(|S| + 2\mathfrak{g} - 2\right),
        \end{equation*}
        where $\mathfrak{g}$ is the genus of $F/\mathbb{C}$. We now estimate the genus and $|S|$ in terms of $\deg h$.
        We start with the genus. In order to use Castelnuovo’s inequality (Proposition \ref{castelnuvothm}), we define
        \[
         F_1 = \mathbb{C}(x, \lambda_1, \lambda_2, \lambda_3), \quad F_2 = \mathbb{C}(y, \mu_1,\mu_2, \mu_3).
        \]
        Note that $\mathbb{C}$ is the field of constants of $F_1, F_2$ and that $F = F_1F_2$. Let $k_i := [F:F_i]$,
        $i=1,2$. Recall that the $\lambda_i$’s and $\mu_i$’s are roots of a monic quadratic polynomial
        and that $[\mathbb{C}(x,y):\mathbb{C}(x)] < \deg h$ by Proposition \ref{nlemma1}. Thus
        \begin{align*}
            k_i&=[F:F_i]\leq [F:\C(x)]\\
            &=[F:\C(x,y)][\C(x,y):\C(x)]<3\deg h.
        \end{align*} For $i=1,2$,
        let $\mathfrak{g}_i$ be the genus of $F_i/\mathbb{C}$. Using \cite[Lemma 6.1]{fuchs2003diophantine}, we get that
        \begin{equation*}
            \mathfrak{g}_i\leq 3^9(\deg \mathfrak{D}+1):=3^9(C_1+1), 
        \end{equation*}
        where $C_1:=\deg \mathfrak{D}=3\deg a+\deg c$. By Castelnuovo’s inequality (Proposition~\ref{castelnuvothm}) we get
        \begin{align*}
            \mathfrak{g} &\leq k_{1} \mathfrak{g}_{1} + k_{2} \mathfrak{g}_{2} + (k_{1}-1)(k_{2}-1)\\
            &<6 \deg h \cdot ~3^9 (C_1+1)+ (3\deg h-1)^2\\
            &< 6\cdot ~3^{9}\deg h^2(C_1+1)+9 \deg h^2\\
            &=3 C_2\deg h^2,
        \end{align*} where $C_2:=2\cdot~3^9(C_1+1)+3$. Next to estimate $|S|$,  let
        \begin{align*}
            S_{1} &= \{ v \in S_{0} : \nu(\lambda_{1}) \neq 0 \ \text{or}\ \nu(\lambda_{2}) \neq 0 \ \text{or}\ \nu(\lambda_{3}) \neq 0 \}, \\
            S_{2} &= \{ v \in S_{0} : \nu(w_{1}) \neq 0 \ \text{or}\ \nu(w_{2}) \neq 0 \ \text{or}\ \nu(w_{3}) \neq 0 \},\\
            S_{3}& = \{ v \in S_{0} : \nu(\mu_{1}) \neq 0 \ \text{or}\ \nu(\mu_{2}) \neq 0 \ \text{or}\ \nu(\mu_{3}) \neq 0 \},\\
            S_{4}& = \{ v \in S_{0} : \nu(t_{1}) \neq 0 \ \text{or}\ \nu(t_{2}) \neq 0 \ \text{or}\ \nu(t_{3}) \neq 0 \}.
        \end{align*}
        Clearly, $|S| \leq |S_{1}| + |S_{2}| + |S_{3}| + |S_{4}| + |S_{\infty}|$. 
Since $[F:\mathbb{C}(x)] < 9 \deg h$ we have $|S_{\infty}| < 9 \deg h$. 
For the other sets, we argue as follows.

Note that the $\lambda_{i}$’s are integral over $\mathbb{C}(x)$ and therefore $\nu(\lambda_{i}) \geq 0$ 
for every finite valuation $\nu$. Thus $\nu(\lambda_1\lambda_2\lambda_3) > 0$ if and only if either 
$\nu(\lambda_1) > 0$ or $\nu(\lambda_2) > 0$ or $\nu(\lambda_3) > 0$. Also, by Vieta’s formulae we have $\lambda_{1}\lambda_{2}\lambda_3 = -c(x)$, where $c(x)$ is given in \eqref{recrel} . Further recall that by Lemma~\ref{lem1}.$(f)$ we have 
$\mathcal{H}(c(x)) = \deg c \cdot \mathcal{H}(x)$ and that 
\[
\sum_{\nu} \max(0,\nu(c(x))) = \mathcal{H}(c(x))
\]
by the sum formula. Then we have
\begin{align*}
    |S_{1}|& = |\{ v \in S_{0} : \nu(\lambda_{1}) > 0 \ \text{or}\ \nu(\lambda_{2}) > 0 \ \text{or}\ \nu(\lambda_{3}) > 0 \}|\\
    &=|\{ v \in S_{0} : \nu(\lambda_{1}\lambda_{2}\lambda_{3}) > 0\}|= |\{ v \in S_{0} : \nu(c(x)) > 0\}|\\
    &\leq \sum_{\nu}{'}1 \leq \sum_{\nu} \max(0,\nu(c(x))) = \mathcal{H}(c(x))\\
    &= \deg c \cdot \mathcal{H}(x) = \deg c \cdot [F:\mathbb{C}(x)] 
< \deg c \cdot 9 \deg h
\end{align*}
where the sum $\sum_{\nu}'$ runs over all valuations $\nu$ 
for which $\nu(c(x)) > 0$ holds.

In order to bound $|S_{3}|$ we argue similarly. We have that $\mu_1,\mu_2,\mu_3$ are the roots 
of the characteristic polynomial of $(W_{n}(y))_{n\geq 0}$, and are hence integral over $\mathbb{C}(y)$. 
Since $y$ is integral over $\mathbb{C}(x)$, we have that $\mu_1,\mu_2,\mu_3$ are integral over $\mathbb{C}(x)$. 
Therefore, as in the case of $S_{1}$ we conclude that $\nu(\mu_{i}) \geq 0$ for every finite valuation $\nu$. 
By Vieta’s formulae we have $\mu_{1}\mu_{2}\mu_3 = -c(y)$. Furthermore, since $h(x) = h(y)$ we have
\[
\deg h \cdot \mathcal{H}(y) = \mathcal{H}(h(y)) = \mathcal{H}(h(x)) = \deg h \cdot \mathcal{H}(x),
\]
and thus
\[
\mathcal{H}(y) = \mathcal{H}(x) = [F:\mathbb{C}(x)].
\]

Similarly,
\begin{equation*}
    |S_{3}| = |\{ v \in S_{0} : \nu(\mu_{1}) > 0 \ \text{or}\ \nu(\mu_{2}) > 0 \ \text{or}\ \nu(\mu_{3}) > 0 \}|<9\cdot \deg c \cdot\deg h
\end{equation*}

For $|S_{2}|$, note that
\[
|S_{2}| \leq |\{ v \in S_{0} : \nu(w_{1}) > 0 \ \text{or}\ \nu(w_{2}) > 0 \ \text{or}\ \nu(w_{3}) > 0 \}|
\]
\[
+ |\{ v \in S_{0} : \nu(w_{1}) < 0 \ \text{or}\ \nu(w_{2}) < 0 \ \text{or}\ \nu(w_{3}) < 0\}|.
\]
To get a bound for this part, we will find a form for $w_1,w_2,w_3$ in terms of the initial values and characteristic roots of the recurrence $W_n(x)$. From the initial conditions
\begin{align*}
 W_0(x) &= w_1(x)+w_2(x)+w_3(x),\\  
 W_1(x) &= w_1(x)\lambda_1(x) + w_2(x)\lambda_2(x) + w_3(x)\lambda_3(x),\\
 W_2(x) &= w_1(x)\lambda_1(x)^2 + w_2(x)\lambda_2(x)^2 + w_3(x)\lambda_3(x)^2, 
\end{align*}
we get
\begin{align}
w_1(x)\Delta(x) &= W_2(x)\big(\lambda_3(x)-\lambda_2(x)\big) 
 + W_1(x)\big(\lambda_2(x)^2-\lambda_3(x)^2\big) \notag\\
&\quad + W_0(x)\,\lambda_2(x)\lambda_3(x)\big(\lambda_3(x)-\lambda_2(x)\big), \label{eqt2.1} \\ 
w_2(x)\Delta(x) &= W_2(x)\big(\lambda_1(x)-\lambda_3(x)\big) 
 + W_1(x)\big(\lambda_3(x)^2-\lambda_1(x)^2\big)\notag\\
&\quad + W_0(x)\,\lambda_1(x)\lambda_3(x)\big(\lambda_1(x)-\lambda_3(x)\big), \label{eqt2.2}\\
w_3(x)\Delta(x) &= W_2(x)\big(\lambda_2(x)-\lambda_1(x)\big) 
 + W_1(x)\big(\lambda_1(x)^2-\lambda_2(x)^2\big) \notag \\
&\quad + W_0(x)\,\lambda_1(x)\lambda_2(x)\big(\lambda_2(x)-\lambda_1(x)\big). \label{eqt2.3}
\end{align}

where
\begin{align*}
\Delta(x) &= \lambda_1(x)\lambda_2(x)\big(\lambda_2(x)-\lambda_1(x)\big)
+ \lambda_1(x)\lambda_3(x)\big(\lambda_1(x)-\lambda_3(x)\big) \\
&\quad + \lambda_2(x)\lambda_3(x)\big(\lambda_3(x)-\lambda_2(x)\big) \\
&= -6i\sqrt{3}\,\sqrt{\mathfrak{D}(x)}.
\end{align*}
That is, we write
\[
w_1=\frac{\ell_1(x)}{\Delta(x)}, ~w_2=\frac{\ell_2(x)}{\Delta(x)}, ~w_3=\frac{\ell_3(x)}{\Delta(x)},
\]
where $\ell_1(x), \ell_2(x), \ell_3(x)$ are the functions in the right hand side of the equations \eqref{eqt2.1},\eqref{eqt2.2}, \eqref{eqt2.3} respectively. Recall that $W_{0}(x), W_{1}(x), W_2(x), \lambda_{1}, \lambda_{2}, \lambda_{3}$ are integral over $\mathbb{C}(x)$, 
and thus also $\ell_1, \ell_2, \ell_3$ and $\Delta(x)$. 
Therefore, for any $\nu \in S_{0}$ we have $\nu(\ell_i)\geq 0$ and $\nu(\Delta(x))\geq 0$. Thus
\begin{align*}
   \nu(w_i)=&\nu\left(\frac{\ell_i(x)}{\Delta(x)}\right)= \nu(\ell_i(x)) + \nu \left( \frac{1}{\Delta(x)} \right)= \underbrace{\nu(\ell_i(x))}_{\geq 0}- \underbrace{\nu(\Delta(x))}_{\geq 0}, \quad i=1,2,3. 
\end{align*}
Hence for $\nu \in S_{0}$ it follows that
\[
\nu(w_{i}) > 0 \ \text{implies}\ \nu(\ell_i(x)) > 0,
\]
\[
\nu(w_{i}) < 0 \ \text{implies}\ \nu(\Delta(x)) > 0, \quad i=1,2,3..
\]

Further note that since $\nu(\ell_i(x)) \geq 0$ for all $i=1,2,3$ and any $\nu \in S_{0}$ we have that either 
$\nu(\ell_1(x)) > 0$ or $\nu(\ell_2(x)) > 0$ or $\nu(\ell_3(x)) > 0$ if and only if
\[
\nu\big(\ell_1(x)\ell_2(x)\ell_3(x)\big) > 0.
\]
Also we can write
\[
\ell_1=(\lambda_3-\lambda_2)[W_2-W_1(\lambda_3+\lambda_2)+W_0\lambda_3\lambda_2]
\]
\[
\ell_2=(\lambda_1-\lambda_3)[W_2-W_1(\lambda_1+\lambda_3)+W_0\lambda_1\lambda_3]
\]
\[
\ell_3=(\lambda_2-\lambda_1)[W_2-W_1(\lambda_2+\lambda_1)+W_0\lambda_2\lambda_1].
\] Now using the Vieta's formulae we get,
\begin{align*}
    \ell_1\ell_2\ell_3=& -\Delta\left( W_2^3+2a W_1W_2^2+bW_0W_2^2+(a^2+b)W_1^2W_2\right.\\
    &\left. +(ab+3c)W_0W_1W_2 +acW_0^2W_2+(ab-c) W_1^3 \right.\\
    & \left. +(b^2+ac)W_1^2W_0+ 2bcW_0^2W_1+cW_0^3\right).
\end{align*}
Thus, 
\[
\mathcal{H}(\ell_1\ell_2\ell_3)=C_3\mathcal{H}(x),
\]
where $C_3(W_i, a,b,c, i=0,1,2)$, a constant depending only on the initial values and coefficients of the recurrence $W_n(x)$.
Therefore,
\begin{align*}
   |S_{2}|& \leq |\{ v \in S_{0} : \nu(w_{1}) > 0 \ \text{or}\ \nu(w_{2}) > 0 \ \text{or}\ \nu(w_{3}) > 0 \}|\\
   &+ |\{ v \in S_{0} : \nu(w_{1}) < 0 \ \text{or}\ \nu(w_{2}) < 0 \ \text{or}\ \nu(w_{3}) < 0\}|.\\
   &\leq |\{ v \in S_{0} : \nu(\ell_1(x)) > 0 \ \text{or}\ 
                     \nu(\ell_2(x)) > 0 \ \text{or}\ 
                     \nu(\ell_3(x)) > 0 \}|\\
   &\quad + |\{ v \in S_{0} : \nu(\Delta(x)) > 0 \}|\\
   & = |\{ v \in S_{0} : \nu(\ell_1(x)\ell_2(x)\ell_3(x)) > 0 \}|+ |\{ v \in S_{0} : \nu(\Delta(x)) > 0 \}|,
\end{align*} and then arguing similarly as for $S_{1}$ we get
\begin{align*}
    |S_{2}| \leq& \mathcal{H}(\ell_1(x)\ell_2(x)\ell_3(x))) 
   + \mathcal{H}(\Delta(x))=\mathcal{H}(\ell_1(x)\ell_2(x)\ell_3(x))) 
   + \frac{1}{2}\mathcal{H}(\mathfrak{D}(x))\\
   =& C_3\mathcal{H}(x)+\frac{1}{2}\deg \mathfrak{D} \mathcal{H}(x)=C_4\mathcal{H}(x) < 9C_4\, \deg h,
\end{align*}
where $C_4=C_3+\frac{1}{2}\deg \mathfrak{D}$. Similarly we argue for $|S_{4}|$ and obtain that
$|S_4|\leq 9C_4 \deg h.$ This gives
\begin{align*}
  |S| & \le |S_1| + |S_2| + |S_3| + |S_4| + |S_\infty|\\ 
&< 18 \deg c\cdot \deg h + 18 C_4 \, \deg h+ 9\, \deg h \\
&= (18 \deg c+18 C_4 +9) \deg h =C_5\, \deg h,
\end{align*} where $C_5:=(18\deg c+18C_4+9).$
Finally we get
\begin{align*}
\mathcal{H}(p_1) &\le 10 ~(2\mathfrak{g} - 2 + |S|) < 10~(6C_2 \, \deg h^2 + C_5\, \deg h- 2) \\
&< 10~(6C_2 + C_5) \deg h^2= C_6 \deg h^2,
\end{align*}
where $C_6:=10(6C_2+C_5)$.
We continue to estimate the terms in \eqref{eqt2.5}. To give an upper bound on $\mathcal{H}(\lambda_i), ~i=1,2,3$ we use \eqref{eqt2.6}, \eqref{eqt2.7}, \eqref{eqt2.8} and obtain the following.
\begin{align*}
    \mathcal{H}(\lambda_1)=&\mathcal{H}(u+v+\frac{1}{3}a) \le \mathcal{H}(u)+\mathcal{H}(v)+\mathcal{H}(\frac{1}{3}a)\\
    =& \mathcal{H}(\sqrt[3]{\frac{q}{2}+\sqrt{\mathfrak{D}}})+\mathcal{H}(\sqrt[3]{\frac{q}{2}-\sqrt{\mathfrak{D}}})+\mathcal{H}(a) \leq  \frac{2}{3}\left(\mathcal{H}(\frac{q}{2})+\frac{1}{2}\mathcal{H}(\mathfrak{D}) \right)+\mathcal{H}(a)\\
    =& \left(\frac{2}{3}(\deg q+ \frac{1}{2}\deg \mathfrak{D})+\deg a\right) \mathcal{H}(x)=C_6\mathcal{H}(x),
\end{align*}
where $C_6:= \frac{2}{3}(\deg q+\frac{1}{2}\deg \mathfrak{D})+\deg a$. Next we bound for $\mathcal{H}(\lambda_2)$.
\begin{align*}
    \mathcal{H}(\lambda_2)=& \mathcal{H}(-\frac{u+v}{2}+ i\sqrt{3}\frac{u-v}{2}+\frac{1}{3}a)
    \leq 2\mathcal{H}(u)+2\mathcal{H}(v)+\mathcal{H}(a)\\
    \leq &2\left(\mathcal{H}(\sqrt[3]{\frac{q}{2}+\sqrt{\mathfrak{D}}})+\mathcal{H}(\sqrt[3]{\frac{q}{2}-\sqrt{\mathfrak{D}}})\right)+\mathcal{H}(a)
    \leq 4\left(\mathcal{H}(\frac{q}{2})+\frac{1}{2}\mathcal{H}(\mathfrak{D}) \right)+\mathcal{H}(a)\\
    =& \left(4(\deg q +\frac{1}{2}\deg \mathfrak{D})+\deg a \right)\mathcal{H}(x)=C_7\mathcal{H}(x),
\end{align*}
where $C_7:=4(\deg q +\frac{1}{2}\deg \mathfrak{D})+\deg a$.
In a similar manner, we obtain
\[
\mathcal{H}(\lambda_3)\leq C_7\mathcal{H}(x).
\]
Thus, 
\[
\mathcal{H}(\lambda_1)+\mathcal{H}(\lambda_2)+\mathcal{H}(\lambda_3)\leq (C_6+2C_7)\mathcal{H}(x).
\]
Since $\mathcal{H}(x)=\mathcal{H}(y)$, we obtain the same upper bound for $\mathcal{H}(\mu_1),\mathcal{H}(\mu_2)$ and $\mathcal{H}(\mu_3)$, that is
\begin{equation*}
    \mathcal{H}(\mu_1)\leq C_6 \mathcal{H}(x) \quad \text{and} \quad \mathcal{H}(\mu_i)\leq C_7 \mathcal{H}(x), \quad i=2,3.
\end{equation*}
Furthermore, we have
\begin{align*}
    \mathcal{H}(w_1)=&~\mathcal{H}\left(\frac{\ell_1(x)}{\Delta(x)}\right)
    \leq \mathcal{H}(\ell_1(x))+\mathcal{H}(\Delta(x))\\
    \leq &~ \mathcal{H}\left((\lambda_3-\lambda_2)[W_2-W_1(\lambda_3+\lambda_2)+W_0\lambda_3\lambda_2]\right)+ \frac{1}{2}\mathcal{H}(\mathfrak{D}(x))\\
    \leq &~ 3\mathcal{H}(\lambda_3)+3\mathcal{H}(\lambda_2)+\mathcal{H}(W_2)+\mathcal{H}(W_1)+\mathcal{H}(W_0)+\frac{1}{2}\mathcal{H}(\mathfrak{D}(x))\\
    \leq &~ \left(12 C_7+\deg W_2+\deg W_1+\deg W_0+\frac{1}{2} \deg \mathfrak{D} \right)\mathcal{H}(x)=C_8 \mathcal{H}(x),
\end{align*}
where $C_8:=12 C_7+\deg W_2+\deg W_1+\deg W_0+\frac{1}{2}\deg \mathfrak{D}$. Similarly we get
\begin{align*}
    \mathcal{H}(w_2)\leq&~ 3\mathcal{H}(\lambda_1)+3\mathcal{H}(\lambda_3)+\mathcal{H}(W_2)+\mathcal{H}(W_1)+\mathcal{H}(W_0)+\frac{1}{2}\mathcal{H}(\mathfrak{D}(x))\\
    \leq &~(3C_6+3C_7+\deg W_2+\deg W_1+\deg W_0+\frac{1}{2} \deg \mathfrak{D})\mathcal{H}(x)\\
    =& ~ C_9\mathcal{H}(x),
\end{align*}
where $C_9:=3C_6+3C_7+\deg W_2+\deg W_1+\deg W_0+\frac{1}{2} \deg \mathfrak{D}$ and
\[
\mathcal{H}(w_3)\leq C_9\mathcal{H}(x).
\] Therefore, 
\begin{equation*}
    \mathcal{H}(w_1)+\mathcal{H}(w_2)+\mathcal{H}(w_3)\leq (C_8+2C_9)\mathcal{H}(x).
\end{equation*}
Next, we estimate the height of $r_1$ as follows:
\begin{align*}
    \mathcal{H}(r_1)=&~ \mathcal{H}\left(\frac{t_1}{t_3}\right)\\
    =&~ \mathcal{H}\left( \frac{(\mu_3-\mu_2)[W_2-W_1(\mu_3+\mu_2)+W_0\mu_3\mu_2]}{(\mu_2-\mu_1)[W_2-W_1(\mu_2+\mu_1)+W_0\mu_2\mu_1]}\right)  \\
    \leq &~ 3\mathcal{H}(\mu_1)+6\mathcal{H}(\mu_2)+3\mathcal{H}(\mu_3)+2\mathcal{H}(W_2)+ 2\mathcal{H}(W_1)+2\mathcal{H}(W_0)\\
    \leq &~ (3C_6+9C_7+2\deg W_2+2\deg W_1+2\deg W_0  )\mathcal{H}(x)=C_{10}\mathcal{H}(x),
\end{align*}
where $C_{10}:=3C_6+9C_7+2\deg W_2+2\deg W_1+2\deg W_0  $. 

We now find a lower bound for $\mathcal{H}(q_1)$ in terms of $\mathcal{H}(y)$. We know that $\mathcal{H}(q_1)\geq 1$. Moreover,
\begin{align*}
   \mathcal{H}(q_1)&=\mathcal{H}\left(\frac{\mu_1}{\mu_2}\right)=\mathcal{H}(\mu_1)+\mathcal{H}(\mu_2), \quad(\text{since $\nu(\mu_1)\nu(\mu_2)<0$})\\ 
   &=\mathcal{H}(u+v+\frac{1}{3}a)+\mathcal{H}(-\frac{u+v}{2}+i\sqrt{3}\frac{u-v}{2}+\frac{1}{3}a)\\
   &=\mathcal{H}(u)+\mathcal{H}(v)+\mathcal{H}(\frac{1}{3}a)+\mathcal{H}\left(-\frac{u+v}{2}+i\sqrt{3}\frac{u-v}{2}+\frac{1}{3}a\right)\\
   &\geq \mathcal{H}(u)+\mathcal{H}(v)+\mathcal{H}(a)+\mathcal{H}\left(-\frac{u+v}{2}+\frac{1}{3}a\right)-\mathcal{H}\left(\frac{u-v}{2}\right)\\
   &= \mathcal{H}(u)+\mathcal{H}(v)+\mathcal{H}(a)+\mathcal{H}(u)+\mathcal{H}(v)+\mathcal{H}(a)-\mathcal{H}(u)-\mathcal{H}(v)\\
   &= \mathcal{H}(u)+\mathcal{H}(v)+2\mathcal{H}(a)=\mathcal{H}\left(\sqrt[3]{\frac{q}{2}+\sqrt{\mathfrak{D}}}\right)+\mathcal{H}\left(\sqrt[3]{\frac{q}{2}-\sqrt{\mathfrak{D}}}\right)+2\mathcal{H}(a)\\
   &=\frac{1}{3}\mathcal{H}\left(\frac{q}{2}+\sqrt{\mathfrak{D}}\right)+\frac{1}{3}\mathcal{H}\left(\frac{q}{2}-\sqrt{\mathfrak{D}}\right)+2\mathcal{H}(a)\\
   &\geq \frac{1}{3}\left(\mathcal{H}\left(\frac{q}{2}\right)-\frac{1}{2}\mathcal{H}(\mathfrak{D})\right)+\frac{1}{3}\left(\mathcal{H}\left(\frac{q}{2}\right)-\frac{1}{2}\mathcal{H}(\mathfrak{D})\right)+2\mathcal{H}(a)\\
   &=\frac{2}{3}\mathcal{H}\left(\frac{q}{2}\right)-\frac{1}{3}\mathcal{H}(\mathfrak{D})+2\mathcal{H}(a)=\left(\frac{2}{3}\deg q-\frac{1}{3}\deg \mathfrak{D}+ 2\deg a\right)\mathcal{H}(y)=C_{11}\mathcal{H}(y).
\end{align*}
  By \eqref{eqndeg} we have $C_{11}:=(\frac{2}{3}\deg q-\frac{1}{3}\deg \mathfrak{D}+ 2\deg a)$ is a positive constant that depends only on $\deg a, \deg b$ and $\deg c$.
Hence we found all the required bounds for the terms in \eqref{eqt2.5} and thus we get
\begin{align}\label{eq29}
    \mathcal{H}(W_n(x))&< \frac{(C_6 \deg h^2+ 2C_{10}9\deg h^2)}{C_{11}9\deg h}\left((C_8+2C_9)9\deg h+ (C_6+2C_7)9\deg h\right)\nonumber\\
    <&\deg h^2 \left(\frac{(C_6+18C_{10})(9C_8+18C_9+9C_6+18C_7)}{9C_{11}}\right):=C_{12}\deg h^2,
\end{align}
where $C_{12}(a,b,c,W_i: i=0,1,2)$ is a positive constant. To give a suitable lower bound for $\mathcal{H}(W_n(x))$, note that since $W_n=g\circ h$ we have 
\begin{equation*}
    \mathcal{H}(W_n(x))=\deg g\deg h~ [F: \C(x)]=\deg g \, \deg h \cdot [F : \mathbb{C}(x,y)] \cdot [\mathbb{C}(x,y) : \mathbb{C}(x)].
\end{equation*}
By \cite[Lemma 3]{fuchs2019decomposable} it follows that $[\mathbb{C}(x,y) : \mathbb{C}(x)] \ge \frac{1}{2} \deg h$. Therefore, we have
\begin{equation}\label{eq30}
    \mathcal{H}(W_n(x)) \ge \frac{1}{2} \deg g \, \deg h^2 \cdot [F : \mathbb{C}(x,y)] \ge \frac{1}{2} \deg g \, \deg h^2.
\end{equation}
Finally from \eqref{eq29} and \eqref{eq30},
\[
\frac{1}{2} \deg g \, \deg h^2 \le \mathcal{H}(W_n(x)) < C_{12} \deg h^2,
\]
and therefore that $\deg g < 2C_{12}$. This finishes the proof of Theorem \ref{2thm}. \qed
\section{Proof of Theorem \ref{thm2}}
Since \eqref{eq5.8} has no proper vanishing subsum, we rewrite \eqref{eq5.8} as 
\begin{equation*}\label{neq2.1}
			1-\frac{\sigma_1\alpha_1^{n_1}}{\delta_2\beta_2^{n_2}}-\frac{\sigma_2\alpha_2^{n_1}}{\delta_2\beta_2^{n_2}}-\frac{\sigma_1\alpha_1^{n_2}}{\delta_2\beta_2^{n_2}}-\frac{\sigma_2\alpha_2^{n_2}}{\delta_2\beta_2^{n_2}}+\frac{\delta_1\beta_1^{n_1}}{\delta_2\beta_2^{n_2}}+\frac{\delta_2\beta_2^{n_1}}{\delta_2\beta_2^{n_2}}+\frac{\delta_1\beta_1^{n_2}}{\delta_2\beta_2^{n_2}}=0.
\end{equation*} Now put
\begin{equation*}
			p_1=-\frac{\sigma_1\alpha_1^{n_2}}{\delta_2\beta_2^{n_2}}, \quad p_2=\frac{\sigma_2\alpha_2^{n_2}}{\delta_2\beta_2^{n_2}}, \quad p_3=\frac{\delta_1\beta_1^{n_2}}{\delta_2\beta_2^{n_2}},
		\end{equation*} and also
		\begin{equation*}
			q_1=\frac{\alpha_1}{\beta_2},~ q_2=\frac{\alpha_2}{\beta_2},~ q_3=~ \frac{\beta_1}{\beta_2},~
			r_1=\frac{\sigma_1}{\delta_2},~r_2=\frac{\sigma_2}{\delta_2}, ~ r_3=\frac{\delta_1}{\delta_2}.
		\end{equation*}
		Consider the field $K=\C(x,y,\alpha_1,\alpha_2,\beta_1,\beta_2)$ and define the projective height $\mathcal{H}$ on $K/\C(x)$ and we obtain the estimate
		\begin{equation*}
			\mathcal{H}\left(\frac{\sigma_i\alpha_i^{n_2}}{\delta_2\beta_2^{n_2}}\right)\geq n_2\cdot \mathcal{H}\left(\frac{\alpha_i}{\beta_2}\right)-\mathcal{H}\left(\frac{\sigma_i}{\delta_2}\right), \quad i=1,2.
		\end{equation*} Similarly we have 
		\begin{equation*}
			\mathcal{H}\left(\frac{\delta_1\beta_1^{n_2}}{\delta_2\beta_2^{n_2}}\right)\geq n_2\cdot \mathcal{H}\left(\frac{\beta_1}{\beta_2}\right)-\mathcal{H}\left(\frac{\delta_1}{\delta_2}\right).
		\end{equation*} That is, $$\mathcal{H}(p_i)\geq n_2 \cdot \mathcal{H}(q_i)-\mathcal{H}(r_i),$$ for $i=1,2,3$. Since $(u_n(y))_{n\geq0}$ is non-degenerate, we have $q_3\notin \C$. It then follows that $\mathcal{H}(q_3)\neq 0$. So,
		\begin{equation}\label{eq5.10}
			n_2\leq (\mathcal{H}(p_3)+\mathcal{H}(r_3))\cdot \mathcal{H}(q_3)^{-1}.
		\end{equation} Alternatively, we derive the following estimate for height of $u_{n_1}(x)+u_{n_2}(x)$;
		\begin{align*}
			\mathcal{H}(u_{n_1}(x)+u_{n_2}(x))&=\mathcal{H}(\sigma_1\alpha_1^{n_1}+\sigma_2\alpha_2^{n_1}+\sigma_1\alpha_1^{n_2}+\sigma_2\alpha_2^{n_2})\\
			&\leq 2\mathcal{H}(\sigma_1)+2\mathcal{H}(\sigma_2)+n_1\mathcal{H}(\alpha_1)+n_1\mathcal{H}(\alpha_2)\\
			&+n_2\mathcal{H}(\alpha_1)+n_2\mathcal{H}(\alpha_2)\\
			&\leq 2n_2 (\mathcal{H}(\sigma_1)+\mathcal{H}(\sigma_2)+\mathcal{H}(\alpha_1)+\mathcal{H}(\alpha_2)).
		\end{align*}
		From \eqref{eq5.10}, we get that
		\begin{equation*}\label{eq21}
		\mathcal{H}(u_{n_1}(x)+u_{n_2}(x))\leq 2(\mathcal{H}(p_3)+\mathcal{H}(r_3))\cdot \mathcal{H}(q_3)^{-1}(\mathcal{H}(\sigma_1) 
		\end{equation*}
        \begin{equation*}
            +\mathcal{H}(\sigma_2)+\mathcal{H}(\alpha_1)+\mathcal{H}(\alpha_2)).
        \end{equation*}
        One can complete the remaining proof by following the arguments as on proof of Theorem \ref{2thm} and hence we omit the details.
By proceeding like the proof of Theorem \ref{2thm} we obtain the  bound as follows.
\[
C = 32 \left(2(\deg A_0 + \deg u_0 + \deg u_1) + 7C_{13} + 2C_{14} + 1\right) (\deg u_0 + \deg u_1 + 6C_{13}),
\] where 
$$C_{13} := \max\{\deg A_0, 2\deg A_1\}$$
and
\[
C_{14} := \max\{2\deg u_{1}, \deg u_{0}+\deg u_{1}+\deg A_{1},\ 2\deg u_{0}+\deg A_{0}\}.
\] This completes the proof of Theorem \ref{thm2}. \qed
\subsection{Proof of Corollary \ref{corolary1}:} Assume that $u_{n_1}(x)+u_{n_2}(x)=v_{m_1}(y)+v_{m_2}(y)$ has infinitely many solutions with $x,y \in \Z$. By Theorem \ref{biluthm}, 
	\begin{equation}\label{eq5.6}
		u_{n_1}(x)+u_{n_2}(x)=\phi(f_1(\lambda(x)))
	\end{equation} and 
	\begin{equation}\label{eq5.7}
		v_{m_1}(x)+v_{m_2}(x)=\phi(g_1(\mu(x))), 
	\end{equation} where $\phi(x)\in \Q[x]$, $\lambda(x),\mu(x)\in \Q[x]$ are linear polynomials, and $(f_1,g_1)$ is a standard pair over $\Q$. Since $u_n$ and $v_m$ are non constant, we have $\phi$ is non constant. Now suppose that $\deg f_1>1$ and $\deg g_1>1$. Then either $u_{n_1}+u_{n_2}$ or $v_{m_1}+v_{m_2}$ is a composite of either cyclic or dihedral polynomial, a contradiction. 
	
	If $\deg g_1=1$ and $\deg f_1=1$ then from \eqref{eq5.7} we get $(v_{m_1}(x)+v_{m_2}(x))\circ \ell_1(x)=\phi(x)$, where $\ell_1(x)$ is a linear polynomial. So then we have, 
    \begin{equation*}
        u_{n_1}(x)+u_{n_2}(x)=(v_{m_1}\circ \ell_1+v_{m_2}\circ \ell_1)\circ f_1 \circ \lambda (x) = v_{m_1}\circ \ell (x)+ v_{m_2}\circ \ell(x),
    \end{equation*} where $\ell(x)\in \Q[x]$ is linear. Now if $\deg g_1=1$ and $\deg f_1>1$, then we have $v_{m_1}(\ell(x))+v_{m_2}(\ell(x))=\phi(x)$, for linear $\ell(x)\in \Q[x]$. So that we obtain 
    \begin{equation}\label{eq25}
        u_{n_1}(x)+u_{n_2}(x)=(v_{m_1}+v_{m_2})(\ell(f_1(\lambda(x)))).
    \end{equation} Since by assumption \eqref{eq25} has no proper vanishing subsum and that $u_{n_1}+u_{n_2}$ is also not a composite of a cyclic or a dihedral polynomial, by Theorem \ref{thm2} we get $\deg (v_{m_1}+v_{m_2})<C$.
    \qed
\section{Concluding Remark}
   In Theorem \ref{2thm} we use assumptions \eqref{eqval13} and \eqref{eqval14} to find a lower bound for $\mathcal{H}(\mu_1/\mu_3)$. One can avoid those assumptions by assuming $\mathcal{H}(\mu_1/\mu_3) > C\cdot \mathcal{H}(y)$ for some constant $C>0$. Furthermore, in Theorem \ref{2thm} and \ref{thm2}, we are assuming that \eqref{thm2eq1} and \eqref{eq5.8} has no proper vanishing subsum, respectively. It will be an interesting problem to find some explicit cases for which there does not exist proper vanishing subsum. 	

{\bf Acknowledgment:} D.N. and S.S.R. are supported by a grant from Anusandhan National Research Foundation (File No.:CRG/2022/000268).

\end{document}